\documentclass[10pt]{amsart}
\usepackage{amssymb,MnSymbol}
\usepackage{amsthm,amsmath}

\usepackage{times}

\title{Computing system signatures through reliability functions}

\author{Jean-Luc Marichal}
\address{Mathematics Research Unit, FSTC, University of Luxembourg, 6, rue Coudenhove-Kalergi, L-1359 Luxembourg, Luxembourg}
\email{jean-luc.marichal[at]uni.lu}

\author{Pierre Mathonet}
\address{University of Li\`ege, Department of Mathematics, Grande Traverse, 12 - B37, B-4000 Li\`ege, Belgium}
\email{p.mathonet[at]ulg.ac.be }

\date{September 14, 2012}

\begin{document}

\theoremstyle{plain}
\newtheorem{theorem}{Theorem}
\newtheorem{lemma}[theorem]{Lemma}
\newtheorem{proposition}[theorem]{Proposition}
\newtheorem{corollary}[theorem]{Corollary}
\newtheorem{fact}[theorem]{Fact}
\newtheorem*{main}{Main Theorem}

\theoremstyle{definition}
\newtheorem{definition}[theorem]{Definition}
\newtheorem{example}[theorem]{Example}

\theoremstyle{remark}
\newtheorem*{conjecture}{Conjecture}
\newtheorem{remark}{Remark}
\newtheorem{claim}{Claim}

\newcommand{\N}{\mathbb{N}}                     
\newcommand{\R}{\mathbb{R}}                     
\newcommand{\Vspace}{\vspace{2ex}}                  
\newcommand{\bfx}{\mathbf{x}}

\begin{abstract}
It is known that the Barlow-Proschan index of a system with i.i.d.\ component lifetimes coincides with the Shapley value, a concept introduced
earlier in cooperative game theory. Due to a result by Owen, this index can be computed efficiently by integrating the first derivatives of the
reliability function of the system along the main diagonal of the unit hypercube. The Samaniego signature of such a system is another important
index that can be computed for instance by Boland's formula, which requires the knowledge of every value of the associated structure function.
We show how the signature can be computed more efficiently from the diagonal section of the reliability function via derivatives. We then apply
our method to the computation of signatures for systems partitioned into disjoint modules with known signatures.
\end{abstract}

\keywords{System signature; reliability function; modular decomposition}

\subjclass[2010]{60K10, 62N05, 90B25}

\maketitle

\section{Introduction}

Consider an $n$-component system $([n],\phi)$, where $[n]=\{1,\ldots,n\}$ is the set of its components and $\phi\colon\{0,1\}^n\to\{0,1\}$ is its structure
function (which expresses the state of the system in terms of the states of its components). We assume that the system is semicoherent, which
means that $\phi$ is nondecreasing in each variable and satisfies the conditions $\phi(0,\ldots,0)=0$ and $\phi(1,\ldots,1)=1$. We also assume, unless otherwise stated, that the components have continuous and i.i.d.\ lifetimes $T_1,\ldots,T_n$.

Barlow and Proschan \cite{BarPro75} introduced in 1975 an index which measures an importance degree for each component. This index is defined by
the $n$-tuple $\mathbf{I}_{\mathrm{BP}}$ whose $k$th coordinate ($k\in [n]$) is the probability that the failure of component $k$
causes the system to fail; that is,
$$
I_{\mathrm{BP}}^{(k)} ~=~ \Pr(T_S=T_k)\, ,
$$
where $T_S$ denotes the system lifetime. For continuous i.i.d.\ component lifetimes, this index reduces to the Shapley value \cite{Sha53,ShaShu54}, a concept
introduced earlier in cooperative game theory. In terms of the values $\phi(A)$ $(A\subseteq [n])$ of the structure function,\footnote{As usual,
we identify Boolean vectors $\bfx\in\{0,1\}^n$ and subsets $A\subseteq [n]$ by setting $x_i=1$ if and only if $i\in A$. We thus use the same
symbol to denote both a function $f\colon\{0,1\}^n\to\R$ and its corresponding set function $f\colon 2^{[n]}\to\R$, interchangeably.} the
probability $I_{\mathrm{BP}}^{(k)}$ then takes the form
\begin{equation}\label{eq:asad67}
I_{\mathrm{BP}}^{(k)}~=~\sum_{A\subseteq [n]\setminus\{k\}}\frac{1}{n\,{n-1\choose |A|}}\,\big(\phi(A\cup\{k\})-\phi(A)\big)\, .
\end{equation}

The concept of \emph{signature}, which reveals a strong analogy with that of Barlow-Proschan index above (see \cite{MarMat} for a recent
comparative study), was introduced in 1985 by Samaniego \cite{Sam85,Sam07} as a useful tool for the analysis of theoretical behaviors of
systems. The system signature is defined by the $n$-tuple $\mathbf{s}$ whose $k$th coordinate $s_k$ is the probability that the $k$th component
failure causes the system to fail. That is,
$$
s_k ~=~ \Pr(T_S=T_{k:n})\, ,
$$
where $T_{k:n}$ denotes the $k$th smallest lifetime, i.e., the $k$th order statistic obtained by rearranging the variables $T_1,\ldots,T_n$ in
ascending order of magnitude.

Boland \cite{Bol01} showed that $s_k$ can be explicitly written in the form
\begin{equation}\label{eq:asad678}
s_k~=~\frac{1}{{n\choose n-k+1}}\,\sum_{\textstyle{A\subseteq [n]\atop |A|=n-k+1}}\phi(A)-\frac{1}{{n\choose n-k}}\,\sum_{\textstyle{A\subseteq
[n]\atop |A|=n-k}}\phi(A)\, .
\end{equation}

Thus, just as for the Barlow-Proschan index, the signature does not depend on the distribution of the variables $T_1,\ldots,T_n$ but only on the
structure function.

The computation of $I_{\mathrm{BP}}^{(k)}$ by means of (\ref{eq:asad67}) may be cumbersome and tedious since it requires the evaluation of
$\phi(A)$ for every $A\subseteq [n]$. To overcome this issue, Owen~\cite{Owe72,Owe88} proposed to compute the right-hand expression in (\ref{eq:asad67}) only from the
expression of $\phi$ as a multilinear polynomial function as follows.

As a Boolean function, $\phi$ can always be put in the unique multilinear form (i.e., of degree at most one in each variable)
$$
\phi(\bfx) ~=~ \sum_{A\subseteq [n]}c(A)\,\prod_{i\in A}x_i\, ,
$$
where the link between the coefficients $c(A)$ and the values $\phi(A)$ is given through the conversion formulas (M\"obius inversion)
$$
c(A)~=~\sum_{B\subseteq A}(-1)^{|A|-|B|}\, \phi(B)\qquad\mbox{and}\qquad \phi(A)~=~\sum_{B\subseteq A}c(B)\, .
$$
Owen introduced the \emph{multilinear extension} of $\phi$ as the multilinear polynomial function $\hat{\phi}\colon [0,1]^n\to\R$ defined by
$$
\hat{\phi}(\bfx) ~=~ \sum_{A\subseteq [n]}c(A)\,\prod_{i\in A}x_i\, .
$$

\begin{example}
The structure of a system consisting of two components connected in parallel is given by
$$
\phi(x_1,x_2) ~=~ \max(x_1,x_2)  ~=~ x_1\amalg x_2  ~=~ x_1+x_2-x_1\, x_2{\,},
$$
where $\amalg$ is the (associative) coproduct operation defined by $x\amalg y=1-(1-x)(1-y)$. Considering only the multilinear expression of $\phi$, we immediately obtain the corresponding multilinear extension $\hat{\phi}(x_1,x_2)=x_1+x_2-x_1\, x_2$.
\end{example}

In reliability analysis the function $\hat{\phi}$, denoted by $h$, is referred to as the \emph{reliability function} of the structure $\phi$
(see \cite[Chap.~2]{BarPro81}; see also \cite[Section 3.2]{Ram90} for a recent reference). This is due to the fact that, under the i.i.d.\ assumption, we have
\begin{equation}\label{eq:sf6f}
\overline{F}_S(t)~=~h\big(\overline{F}_1(t),\ldots,\overline{F}_n(t)\big)\, ,
\end{equation}
where $\overline{F}_S(t)=\Pr(T_S>t)$ is the reliability of the system and $\overline{F}_k(t)=\Pr(T_k>t)$ is the reliability of component $k$ at
time $t$.

We henceforth denote the function $\hat{\phi}$ by $h$. Also, for any function $f$ of $n$ variables, we denote its diagonal section
$f(x,\ldots,x)$ simply by $f(x)$.

Owen then observed that the $k$th coordinate of the Shapley value, and hence the $k$th coordinate of the Barlow-Proschan index, is also given by
\begin{equation}\label{eq:sd5f}
I_{\mathrm{BP}}^{(k)} ~=~ \int_0^1(\partial_k\hat{\phi})(x)\, dx ~=~ \int_0^1(\partial_kh)(x)\, dx\, .
\end{equation}
That is, $I_{\mathrm{BP}}^{(k)}$ is obtained by integrating over $[0,1]$ the diagonal section of the $k$th partial derivative of $h$.

Thus, formula (\ref{eq:sd5f}) provides a simple way to compute $I_{\mathrm{BP}}^{(k)}$ from the reliability function $h$ (at least simpler than the use of
(\ref{eq:asad67})).

\begin{example}\label{ex:5s7d6f}
Consider the bridge structure as indicated in Figure~\ref{fig:bs}. The corresponding structure function and its reliability function are respectively given by
$$
\phi(x_1,\ldots,x_5) ~=~ x_1\, x_4\amalg x_2\, x_5\amalg x_1\, x_3\, x_5\amalg x_2\, x_3\, x_4
$$
and
\begin{eqnarray*}
h(x_1,\ldots,x_5) &=& x_1 x_4 + x_2 x_5 + x_1 x_3 x_5 + x_2 x_3 x_4 \\
&& \null - x_1 x_2 x_3 x_4 - x_1 x_2 x_3 x_5 - x_1 x_2 x_4 x_5  - x_1 x_3 x_4 x_5 - x_2 x_3 x_4 x_5 + 2\, x_1 x_2 x_3 x_4 x_5\, .
\end{eqnarray*}
By using (\ref{eq:sd5f}) we obtain
$\mathbf{I}_{\mathrm{BP}}=\big(\frac{7}{30},\frac{7}{30},\frac{1}{15},\frac{7}{30},\frac{7}{30}\big)$. Indeed, we have for instance
$$
I_{\mathrm{BP}}^{(3)} ~=~  \int_0^1(\partial_3h)(x)\, dx ~=~ \int_0^1(2x^2-4x^3+2x^4)\, dx\, ~=~ \textstyle{\frac{1}{15}}{\,}.
$$\qed
\end{example}

\setlength{\unitlength}{4ex}
\begin{figure}[htbp]\centering
\begin{picture}(11,4)
\put(3,0.5){\framebox(1,1){$2$}} \put(3,2.5){\framebox(1,1){$1$}} \put(5,1.5){\framebox(1,1){$3$}} \put(7,0.5){\framebox(1,1){$5$}}
\put(7,2.5){\framebox(1,1){$4$}}%
\put(0,2){\line(1,0){1.5}}\put(1.5,2){\line(2,-1){1.5}}\put(5.5,0){\line(-2,1){1.5}}\put(1.5,2){\line(2,1){1.5}}\put(5.5,4){\line(-2,-1){1.5}}%
\put(0,2){\circle*{0.15}}%
\put(9.5,2){\line(1,0){1.5}}\put(5.5,0){\line(2,1){1.5}}\put(9.5,2){\line(-2,-1){1.5}}\put(5.5,4){\line(2,-1){1.5}}\put(9.5,2){\line(-2,1){1.5}}%
\put(11,2){\circle*{0.15}}%
\put(5.5,0){\line(0,1){1.5}}\put(5.5,4){\line(0,-1){1.5}}
\end{picture}
\caption{Bridge structure} \label{fig:bs}
\end{figure}
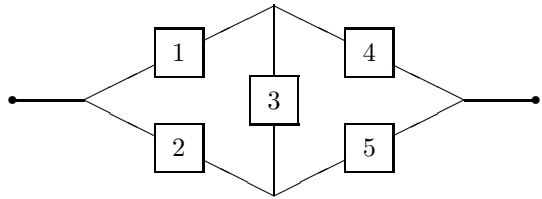

\begin{remark}
Example~\ref{ex:5s7d6f} illustrates the fact that the reliability function $h$ can be easily obtained from the minimal path sets\footnote{Recall that a subset $P\subseteq [n]$ of components is a \emph{path set} for the function $\phi$ if $\phi(P)=1$. A path set $P\subseteq [n]$ is said to be \emph{minimal} if it does not strictly contain another path set.} of the system simply by expanding the coproduct in $\phi$ and simplifying the resulting algebraic expression (using $x_i^2=x_i$).
\end{remark}

Similarly to Owen's method, in this note we provide a way to compute the signature of the system only from the reliability function of the structure,
thus avoiding Boland's formula (\ref{eq:asad678}) which requires the evaluation of $\phi(A)$ for every $A\subseteq [n]$.

Specifically, considering the \emph{tail signature} of the system, that is, the $(n+1)$-tuple
$\overline{\mathbf{S}}=(\overline{S}_0,\ldots,\overline{S}_n)$ defined by (see (\ref{eq:asad678}))
\begin{equation}\label{6w7er5}
\overline{S}_k ~=~ \sum_{i=k+1}^ns_i ~=~ \frac{1}{{n\choose n-k}}\,\sum_{|A|=n-k}\phi(A)\, ,
\end{equation}
we prove (see Theorem~\ref{thm:main} below) that the coefficient of $(x-1)^k$ in the Taylor expansion about $x=1$ of the polynomial
$$
p(x) ~=~ x^n h(1/x)
$$
(which is the $n$-reflected of the univariate polynomial $h(x)$) is exactly ${n\choose k}\,\overline{S}_k$.\footnote{Equivalently, ${n\choose
k}\,\overline{S}_k$ is the coefficient of $x^k$ in $p(x+1)$.}  In other terms, we have
\begin{equation}\label{eq:ds56}
\overline{S}_k ~=~ \frac{(n-k)!}{n!}\, D^kp(1)\, ,\qquad k=0,\ldots,n\, ,
\end{equation}
and the signature can be computed by
\begin{equation}\label{eq:ds56x}
s_k ~=~ \overline{S}_{k-1}-\overline{S}_k\, ,\qquad k=1,\ldots,n\, .
\end{equation}

Even though such a computation can be easily performed by hand for small $n$, a computer algebra system can be of great assistance for large
$n$.

\begin{example}
Consider again the bridge structure as indicated in Figure~\ref{fig:bs}. By identifying the variables $x_1,\ldots,x_5$ in $h(x_1,\ldots,x_5)$, we immediately obtain
$$
h(x) ~=~ 2x^2+2x^3-5x^4+2x^5{\,},
$$
from which we can compute
$$
p(x) ~=~ x^5 h(1/x) ~=~ 2-5x+2x^2+2x^3 ~=~  1+5 (x-1)+8 (x-1)^2+2 (x-1)^3{\,},
$$
or equivalently,
$$
p(x+1) ~=~ 1+5x+8x^2+2x^3{\,}.
$$
Using (\ref{eq:ds56}) we then easily obtain $\overline{\mathbf{S}}=\big(1,1,\frac{4}{5},\frac{1}{5},0,0\big)$. Indeed, we have for instance ${5\choose 2}\,\overline{S}_2=8$ and hence $\overline{S}_2=4/5$. Finally, using (\ref{eq:ds56x}) we obtain $\mathbf{s}=\big(0,\frac{1}{5},\frac{3}{5},\frac{1}{5},0\big)$.\qed
\end{example}

This note is organized as follows. In Section 2 we give a proof of our result by first showing a link between the reliability function and the
tail signature through the so-called Bernstein polynomials. In Section 3 we apply our result to the computation of signatures for systems
partitioned into disjoint modules with known signatures.

\section{Notation and main results}

Recall that the $n+1$ \emph{Bernstein polynomials} of degree $n$ are defined on the real line by
$$
b_{k,n}(x)~=~{n\choose k}\, x^k\, (1-x)^{n-k},\qquad k=0,\ldots,n\, .
$$
These polynomials form a basis of the vector space $P_n$ of polynomials of degree at most $n$.

\begin{proposition}\label{prop:bernstein}
We have
\begin{equation}\label{eq:bernstein}
h(x) ~=~ \sum_{k=0}^n \overline{S}_{n-k}\, b_{k,n}(x).
\end{equation}
Thus, the numbers $\overline{S}_{n-k}$ $(k=0,\ldots,n)$ are precisely the components of the diagonal section of the reliability function $h$ in
the basis formed by the Bernstein polynomials of degree $n$.
\end{proposition}

\begin{proof}
The reliability function can be expressed as
$$
h(\bfx) ~=~ \sum_{A\subseteq [n]}\phi(A)\,\prod_{i\in A}x_i\,\prod_{i\in [n]\setminus A}(1-x_i)\, .
$$
Its diagonal section is then given by
$$
h(x) ~=~ \sum_{A\subseteq [n]}\phi(A)\, x^{|A|}\,(1-x_i)^{n-|A|} ~=~ \sum_{k=0}^n \,\bigg(\sum_{|A|=k}\phi(A)\bigg)\, x^k\,(1-x)^{n-k}
$$
and we immediately conclude by (\ref{6w7er5}).
\end{proof}

By applying the classical transformations between power and Bernstein polynomial forms to Eq.~(\ref{eq:bernstein}), from the standard form of $h(x)$, namely $h(x)=\sum_{k=0}^n a_k\, x^k$, we immediately obtain
\begin{equation}\label{eq:bernstein2}
\overline{S}_k ~=~ \sum_{i=0}^{n-k}\frac{{n-k\choose i}}{{n\choose i}}\, a_i\quad\mbox{and}\quad a_k ~=~ {n\choose k}\,\sum_{i=0}^k(-1)^{k-i}\,{k\choose i}\,\overline{S}_{n-i}\, ,\qquad k=0,\ldots,n.
\end{equation}

\begin{remark}\label{rem:87fd}
\begin{enumerate}
\item[(a)] Eqs.~(\ref{eq:bernstein}) and (\ref{eq:bernstein2}) explicitly show that $h(x)$ encodes exactly the signature, no more, no less. This means that two $n$-component systems having the same $h(x)$ also have the same signature and two $n$-component systems having the same signature also have the same $h(x)$.

It is also noteworthy that two distinct $n$-component systems may have the same $h(x)$, and hence the same signature. For instance, the $8$-component system defined by the structure
$$
\phi_1(\bfx) ~=~ x_1\, x_2\amalg x_2\, x_3\, x_4\amalg x_5\, x_6\, x_7\, x_8
$$
has the same $h(x)$ as the $8$-component system defined by the structure
$$
\phi_2(\bfx) ~=~ x_1\, x_3\amalg x_2\, x_4\, x_5\amalg x_1\, x_2\, x_6\, x_7\, x_8\, ,
$$
namely $h(x)=x^2+x^3-x^6-x^7+x^8$.

\item[(b)] Eq.~(\ref{eq:bernstein}) also shows that $\overline{S}_k$ is the component of $h(x)$ along the basis polynomial $b_{n-k,n}$.
Interestingly, by replacing $x$ by $1-x$ in (\ref{eq:bernstein}), we obtain the following (dual) basis decomposition
$$
h(1-x) ~=~ \sum_{k=0}^n \overline{S}_{n-k}\, b_{n-k,n}(x)~=~ \sum_{k=0}^n \overline{S}_{k}\, b_{k,n}(x)\, .
$$
\item[(c)] Using summation by parts in Eq.~(\ref{eq:bernstein}), we derive the following identity
$$
h(x) ~=~ \sum_{k=1}^ns_k\, h_{\mathrm{os}_k}(x)\, ,
$$
where $h_{\mathrm{os}_k}(x)=\sum_{i=n-k+1}^nb_{i,n}(x)$ is the diagonal section of the reliability function of the $(n-k+1)$-out-of-$n$ system (the structure $\mathrm{os}_k(\bfx)$ being the $k$th smallest variable $x_{k:n}$). By (\ref{eq:sf6f}) we see that
this identity is nothing other than the classical signature-based expression of the system reliability (see, e.g., \cite{Sam07}), that is,
$$
\Pr(T_S>t) ~=~ \sum_{k=1}^ns_k\, \Pr(T_{k:n}>t)\, .
$$
\end{enumerate}
\end{remark}

We can now state and prove our main result. Let $f$ be a univariate polynomial of degree $m\leqslant n$,
$$
f(x)~=~a_n\, x^n+\cdots + a_1\, x+ a_0\, .
$$
The $n$-\emph{reflected polynomial} of $f$ is the polynomial $f^R$ defined by
$$
f^R(x)~=~a_0\, x^n+a_1\, x^{n-1}+\cdots + a_n\, ,
$$
or equivalently, $f^R(x)=x^n\, f(1/x)$.

\begin{theorem}\label{thm:main}
We have
$$
h^R(x) ~=~ \sum_{k=0}^n {n\choose k}\, \overline{S}_k\, (x-1)^k.
$$
Thus, for every $k\in\{0,\ldots,n\}$, the number ${n\choose k}\, \overline{S}_k$ is precisely the coefficient of $(x-1)^k$ of the Taylor
expansion about $x=1$ of the $n$-reflected diagonal section of the reliability function $h$.
\end{theorem}

\begin{proof}
By Proposition~\ref{prop:bernstein}, we have
$$
h(x) ~=~ \sum_{k=0}^n \overline{S}_{k}\, b_{n-k,n}(x).
$$
The result then follows by reflecting this polynomial.
\end{proof}

From Theorem~\ref{thm:main} we immediately derive the following algorithm, which inputs both the number $n$ of components and the reliability
function $h$ and outputs the signature $\mathbf{s}$ of the system.

\smallskip

\begin{quotation}
\begin{enumerate}
\item[\textbf{Step 1.}] Express the $n$-reflected polynomial $h^R(x)=x^n\, h(1/x)$ in the basis $\{(x-1)^k:k=0,\ldots,n\}$ or, equivalently, the polynomial
$h^R(x+1)$ in the basis $\{x^k:k=0,\ldots,n\}$. That is,
$$
h^R(x)~=~\sum_{k=0}^n c_k\, (x-1)^k\qquad\mbox{or}\qquad h^R(x+1)~=~\sum_{k=0}^n c_k\, x^k\, .
$$


\item[\textbf{Step 2.}] Compute the tail signature $\overline{\mathbf{S}}\,$:
$$
\overline{S}_k ~=~ \textstyle{c_k/{n\choose k}}\, ,\qquad k=0,\ldots,n.
$$

\item[\textbf{Step 3.}] Compute the signature $\mathbf{s}\,$:
$$
s_k ~=~ \overline{S}_{k-1}-\overline{S}_k\, ,\qquad k=1,\ldots,n.
$$
\end{enumerate}
\end{quotation}

\smallskip


\begin{remark}
The concept of signature was recently extended to the general non-i.i.d.\ case (see, e.g., \cite{MarMatb}). In fact, assuming only that ties have null probability (i.e., $\Pr(T_i=T_j)=0$ for $i\neq j$), we can define the \emph{probability signature} of the system as the $n$-tuple $\mathbf{p}=(p_1,\ldots,p_n)$, where $p_k=\Pr(T_S=T_{k:n})$. This $n$-tuple may depend on both the structure and the distribution of lifetimes. It was proved \cite{MarMatb} that in general we have
\begin{equation}\label{eq:c8v76c}
\sum_{i=k+1}^n\Pr(T_S=T_{i:n})~=~\sum_{|A|=n-k}q(A)\,\phi(A){\,},
\end{equation}
where the function $q\colon 2^{[n]}\to \R$, called the relative quality function associated with the system, is defined by $q(A)=\Pr(\max_{i\in [n]\setminus A}T_i<\min_{i\in A}T_i)$.

Clearly, the right-hand side of (\ref{eq:c8v76c}) coincides with that of (\ref{6w7er5}) for every semicoherent system when $q(A)=1/{n\choose |A|}$ for every $A\subseteq [n]$ (see \cite{MarMatc} for more details).
Therefore the algorithm above can be applied to the non-i.i.d.\ case whenever this condition holds, for instance when the lifetimes are exchangeable.
\end{remark}

An $n$-component semicoherent system is said to be \emph{coherent} if it has only relevant components, i.e., for every $k\in [n]$ there exists
$\bfx\in\{0,1\}^n$ such that $\phi(0_k,\bfx)\neq\phi(1_k,\bfx)$, where $\phi(z_k,\bfx)=\phi(\bfx)|_{x_k=z}$.

The following proposition gives sufficient conditions on the signature for a semicoherent system to be coherent.

\begin{proposition}\label{prop:8g6}
Let $([n],\phi)$ be an $n$-component semicoherent system with continuous i.i.d.\ component lifetimes. Then the following assertions are equivalent.
\begin{enumerate}
\item[$(i)$] The reliability function $h$ is a polynomial of degree $n$ (equivalently, $h(x)$ is a polynomial of degree $n$).

\item[$(ii)$] We have
$$
\sum_{\mbox{$k$ odd}} {n\choose k}\,\overline{S}_k ~\neq ~\sum_{\mbox{$k$ even}} {n\choose k}\,\overline{S}_k\, .
$$

\item[$(iii)$] We have
$$
\sum_{\mbox{$k$ odd}} {n-1\choose k-1}\, s_k~\neq ~\sum_{\mbox{$k$ even}} {n-1\choose k-1}\, s_k\, .
$$
\end{enumerate}
If any of these conditions is satisfied, then the system is coherent.
\end{proposition}

\begin{proof}
The equivalence $(i)\Leftrightarrow (ii)$ immediately follows from Theorem~\ref{thm:main} and the fact that $h(x)$ is of degree $n$ if and only
if $h^R(0)\neq 0$.

The equivalence $(ii)\Leftrightarrow (iii)$ follows from the straightforward identity
$$
\sum_{k=0}^n{n\choose k}\,\overline{S}_k\, (-1)^k ~=~ \sum_{k=1}^n {n-1\choose k-1}\, s_k\, (-1)^{k-1}\, .
$$

To see that the system is coherent when condition $(i)$ is satisfied, suppose that component $k$ is irrelevant. Then $h(\bfx)=h(1_k,\bfx)$ has
less than $n$ variables and therefore cannot be of degree $n$.
\end{proof}

\begin{remark}\label{rem:sdf876}
\begin{enumerate}
\item[(a)] The equivalent conditions in Proposition~\ref{prop:8g6} are not necessary for a semicoherent system to be coherent. For instance, the $4$-component coherent system defined by the structure
$$
\phi(\bfx) ~=~ x_1\, x_2\amalg x_2\, x_3\amalg x_3\, x_4 ~=~ x_1 x_2 + x_2 x_3  + x_3 x_4 - x_1 x_2 x_3 - x_2 x_3 x_4
$$
has a reliability function of degree $3$.

\item[(b)] The $6$-component coherent system defined by the structure
$$
\phi_1(\bfx) ~=~ x_1\, x_2\amalg x_2\, x_3\, x_4\amalg x_3\, x_4\, x_5\, x_6
$$
has the same $h(x)$ as the $5$-component coherent system (or $6$-component noncoherent system) defined by the structure
$$
\phi_2(\bfx) ~=~ x_1\, x_3\amalg x_2\, x_4\, x_5\, ,
$$
namely $h(x)=x^2+x^3-x^5$. We thus retrieve the fact that $h(x)$ does not characterize the system (see Remark~\ref{rem:87fd}(a)) and cannot determine whether or not the system is coherent (see Remark~\ref{rem:sdf876}(a)).
\end{enumerate}
\end{remark}

\section{Application: Modular decomposition of system signatures}

We now apply our main result to show that (and how) the signature of a system partitioned into disjoint modules can be computed only from the
partition structure and the module signatures.

Suppose that the system is partitioned into $r$ disjoint semicoherent modules $(A_j,\chi_j)$ $(j=1,\ldots,r$), where $A_j$ represents the set of the
components in module $j$ and $\chi_j\colon\{0,1\}^{A_j}\to\{0,1\}$ is the corresponding structure function. Let $n_j$ denote the number of
components in $A_j$ (hence $\sum_{j=1}^r n_j=n$) and let $\overline{\mathbf{S}}_j=(\overline{S}_{j,0},\ldots,\overline{S}_{j,n_j})$ denote the
tail signature of module $j$.

If $\psi\colon\{0,1\}^r\to\{0,1\}$ is the structure function of the partition of the system into modules, the modular decomposition of the
structure $\phi$ of the system expresses through the composition
$$
\phi(\bfx) ~=~ \psi\big(\chi_1(\bfx^{A_1}),\ldots,\chi_r(\bfx^{A_r})\big)\, ,
$$
where $\bfx^{A_j}=(x_i)_{i\in A_j}$ (see \cite[Chap.~1]{BarPro81}). Since the modules are disjoint, this
composition extends to the reliability functions $h_{\phi}$, $h_{\psi}$, and $h_{\chi_j}$ of the structures $\phi$, $\psi$, and $\chi_j$,
respectively; that is,
\begin{equation}\label{eq:asd46}
h_{\phi}(\bfx) ~=~ h_{\psi}\big(h_{\chi_1}(\bfx^{A_1}),\ldots,h_{\chi_r}(\bfx^{A_r})\big)\, .
\end{equation}
Indeed, the right-hand side of (\ref{eq:asd46}) contains no powers and hence is a multilinear polynomial.

According to Theorem~\ref{thm:main}, the tail signature of the system can be computed directly from the function
\begin{equation}\label{eq:asd1}
h_{\phi}^R(x) ~=~ x^n\,h_{\psi}\big(x^{-n_1}\, h_{\chi_1}^R(x),\ldots,x^{-n_r}\,h_{\chi_r}^R(x)\big)\, ,
\end{equation}
where $h_{\chi_j}^R$ is the $n_j$-reflected of the diagonal section of $h_{\chi_j}$, that is,
\begin{equation}\label{eq:asd2}
h_{\chi_j}^R(x) ~=~ \sum_{k=0}^{n_j} {n_j\choose k}\, \overline{S}_{j,k}\, (x-1)^k.
\end{equation}

Interestingly, Eqs.~(\ref{eq:asd1}) and (\ref{eq:asd2}) show that (and how) the signature of the system can be computed only from the structure
$\psi$ and the signature of every module. Thus, the complete knowledge of the structures $\chi_1,\ldots,\chi_r$ is not needed in the computation
of the signature of the whole system.

\begin{example}
Consider a $7$-component system consisting of two serially connected modules (hence $\psi(z_1,z_2)=z_1z_2$) with signatures
$\mathbf{s}_1=\big(\frac{1}{3},\frac{2}{3},0\big)$ and $\mathbf{s}_2=\big(0,\frac{2}{3},\frac{1}{3},0\big)$, respectively. By (\ref{eq:asd2}) we
have
$$
h_{\chi_1}^R(x)=2x-1\quad\mbox{and}\quad h_{\chi_2}^R(x)=2x^2-1\, .
$$
By (\ref{eq:asd1}) we then obtain
$$
h_{\phi}^R(x) ~=~ x^7\big(x^{-3}(2x-1)\,x^{-4}(2x^2-1)\big) ~=~ 1-2x-2x^2+4x^3\, ,
$$
from which we derive the system signature $\mathbf{s}=\big(\frac{1}{7},\frac{8}{21},\frac{38}{105},\frac{4}{35},0,0,0\big)$.\qed
\end{example}

As an immediate consequence of our analysis we retrieve the fact (already observed in \cite{MarMatSpi}; see \cite{DaZheHu12,GerShpSpi11} for
earlier references) that the signature always decomposes through modular partitions.\footnote{This feature reveals an analogy for instance with
the barycentric property of mean values: The arithmetic (geometric, harmonic, etc.) mean of $n$ real numbers does not change when one modifies
some of the numbers without changing their arithmetic (geometric, harmonic, etc.) mean.} We state this property as follows.


\begin{theorem}
The signature of a system partitioned into disjoint modules does not change when one modifies the modules without changing their signatures.
\end{theorem}

A \emph{recurrent system} is a system partitioned into identical modules. Thus, for any recurrent system we have $n_1=\cdots = n_r=n/r$ and
$\chi_1=\cdots =\chi_r=\chi$. Eq.~(\ref{eq:asd1}) then reduces to
\begin{equation}\label{eq:dfg6}
h_{\phi}^R(x) ~=~ x^n\,h_{\psi}\big(x^{-n/r}\,h_{\chi}^R(x)\big) ~=~ h_{\chi}^R(x)^r\, h_{\psi}^R\big(x^{n/r}\, h_{\chi}^R(x)^{-1}\big)\, .
\end{equation}
Thus, to compute the tail signature $\overline{\mathbf{S}}$ of the whole system the knowledge of the structures $\psi$ and $\chi$ can be simply
replaced by the knowledge of their corresponding tail signatures $\overline{\mathbf{S}}_{\psi}$ and $\overline{\mathbf{S}}_{\chi}$,
respectively.

\begin{example}
Consider a system partitioned into $r$ modules, each of whose consists of two components connected in parallel (system with componentwise
redundancy). In this case we have $h_{\chi}^R(x)=2x-1$ and
$$
h_{\psi}^R(x) ~=~ \sum_{k=0}^r{r\choose k}\,\overline{S}_{\psi,k}\, (x-1)^k\, .
$$
By (\ref{eq:dfg6}) it follows that
$$
h_{\phi}^R(x+1) ~=~ \sum_{k=0}^r{r\choose k}\,\overline{S}_{\psi,k}\, x^{2k}\, (2x+1)^{r-k}\, ,
$$
from which we derive
$$
\overline{S}_{\ell} ~=~ \sum_{k=\max({\ell}-r,0)}^{\lfloor\ell/2\rfloor}\frac{{r\choose k}{r-k\choose\ell -2k}}{{2r\choose\ell}}\,2^{\ell
-2k}\,\overline{S}_{\psi,k}{\,},\qquad 0\leqslant\ell\leqslant 2r{\,}.
$$
\end{example}

\section*{Acknowledgments}

The authors wish to thank Miguel Couceiro for fruitful discussions. This research is partly supported by the internal research project F1R-MTH-PUL-12RDO2 of the University of Luxembourg.

\end{document}